\documentclass{amsart}

\usepackage{amssymb}
\usepackage{amsrefs}
\usepackage{amsmath}
\usepackage{amsfonts}
\usepackage{amsthm}
\usepackage{array}
\usepackage{bm}
\usepackage{enumitem}
\usepackage{xcolor}

\usepackage{eucal}
\usepackage{upgreek}

\usepackage{hyperref}

\usepackage{tikz}
\usetikzlibrary{cd}

\hypersetup{
	colorlinks,
	linkcolor={red!50!black},
	citecolor={blue!50!black},
	urlcolor={blue!80!black}
}

\newtheorem{thm}{Theorem}[section]

\newtheorem{cor}[thm]{Corollary}
\newtheorem*{cor*}{Corollary}
\newtheorem*{thm*}{Theorem}

\newtheorem{definition}[thm]{Definition}

\theoremstyle{definition}

\theoremstyle{plain} % just in case the style had changed
\newcommand{\thistheoremname}{}
\newtheorem{genericthm}[thm]{\thistheoremname}

\newtheorem*{genericthm*}{\thistheoremname}
\newenvironment{namedthm*}[1]
{\renewcommand{\thistheoremname}{#1}%
	\begin{genericthm*}}
	{\end{genericthm*}}

\theoremstyle{remark}

\newtheorem{remark}[thm]{Remark}

\numberwithin{equation}{section}

\newcommand{\R}{\mathbb{R}}
\newcommand{\C}{\mathbb{C}}

\newcommand{\SO}{\mathrm{SO}}

\newcommand{\G}{\mathrm{G}}
\newcommand{\U}{\mathrm{U}}
\newcommand{\eps}{\varepsilon}

\renewcommand{\Re}{\operatorname{Re}}
\renewcommand{\Im}{\operatorname{Im}}

\pdfoutput=1

\begin{document}

\title{Static solutions to symplectic curvature flow in dimension four}

\author{Gavin Ball}
\address{\textsc{University of Wisconsin-Madison},
	\textsc{Department of Mathematics},
	\textsc{Madison, WI, USA}}
\email{gball3@wisc.edu}
\urladdr{https://www.gavincfball.com/}

\begin{abstract}
This article studies special solutions to symplectic curvature flow in dimension four. Firstly, we derive a local normal form for static solutions in terms of holomorphic data and use this normal form to show that every complete static solution to symplectic curvature flow in dimension four is K\"ahler-Einstein. Secondly, we perform an exterior differential systems analysis of the soliton equation for symplectic curvature flow and use the Cartan-K\"ahler theorem to prove a local existence and generality theorem for solitons.
\end{abstract}

\maketitle

\tableofcontents

%%%%%%%%%%%%%%%%%%%%%%%%%%%%%%%%%%%%%%%%%%%%%%%%%%%%%%%%%%
\section{Introduction}
%%%%%%%%%%%%%%%%%%%%%%%%%%%%%%%%%%%%%%%%%%%%%%%%%%%%%%%%%%

An almost-K\"ahler manifold $(X, \Omega, J)$ consists of a even-dimensional manifold $X$ endowed with a symplectic form $\Omega$ and a compatible almost complex structure $J.$ Together, $\Omega$ and $J$ define a Riemannian metric $g$ on $X.$ One perspective on almost-K\"ahler geometry is to fix a symplectic form $\Omega$ on $X,$ choose a compatible $J,$ and think of $J$ and $g$ as auxiliary tools used to study the symplectic geometry of $(X, \Omega)$. In this direction, symplectic curvature flow is a degenerate parabolic evolution equation for almost-K\"ahler structures introduced by Streets--Tian \cite{StreetsTian14} given by
\begin{equation}\label{eq:introsympcurvfl}
	\begin{aligned}
		\tfrac{\partial}{\partial t} \Omega &= - 2 \rho, \\
		\tfrac{\partial}{\partial t} g &= - 2 \rho^{1,1} - 2 \, \mathrm{Ric}^{2,0+0,2},
	\end{aligned}
\end{equation}
where $\rho$ is the Chern-Ricci form of $(\Omega, J)$ and $\mathrm{Ric}$ is the Ricci tensor of $g.$ The flow (\ref{eq:introsympcurvfl}) preserves the symplectic condition $d \Omega = 0$ and restricts to K\"ahler-Ricci flow in the case where $J$ is integrable.

The idea behind symplectic curvature flow is to evolve an initial almost-K\"ahler structure on $X$ towards some canonical structure. In their initial paper, Streets \& Tian prove short time existence for (\ref{eq:introsympcurvfl}), but it is to be expected that symplectic curvature flow will encounter singularities in general. Analogy with other geometric flows, especially Ricci flow, suggests that singularity formation should be modeled on \emph{soliton} solutions to (\ref{eq:introsympcurvfl}).

\begin{definition}
	An almost-K\"ahler manifold $(X, \Omega, J, g)$ is a \emph{symplectic curvature flow soliton} if there exists a constant $\lambda \in \R$ and a vector field $V$ on $X$ such that
	\begin{equation}\label{eq:intsympcurvsoliton}
		\begin{aligned}
			 \lambda \Omega + \mathcal{L}_V \Omega &= - 2 \rho, \\
			\lambda g + \mathcal{L}_V g &= - 2 \rho^{1,1} - 2 \, \mathrm{Ric}^{2,0+0,2}.
		\end{aligned}
	\end{equation}
	If $\lambda > 0, $ $ = 0,$ or $ < 0$, we say the soliton is \emph{expanding, steady,} or \emph{shrinking} respectively. If the vector field $V$ vanishes identically, then we say $(X, \Omega, J, g)$ is a \emph{static solution} to symplectic curvature flow.
\end{definition}

The soliton solutions are precisely the almost-K\"ahler structures which evolve by rescaling and diffeomorphisms along (\ref{eq:introsympcurvfl}). The static solutions are the structures which evolve purely by rescaling. For Ricci flow, the static solutions are the Einstein metrics, and so static solutions to symplectic curvature flow can be thought of as the almost-K\"ahler analogues of Einstein metrics.

\subsection{Results}
In this paper, we study soliton and static solutions to symplectic curvature flow in dimension four using the techniques of exterior differential systems and the moving frame. These techniques are well-suited to the non-linear, overdetermined PDE system (\ref{eq:intsympcurvsoliton}).

Our first main result is Theorem \ref{thm:locnormalform}, which gives a local normal form for static solutions. We summarize it here as:

\begin{thm}
	Every four-dimensional static solution to symplectic curvature flow is steady ($\lambda = 0$). If $(X^4, \Omega, J, g)$ is a static solution and $p \in X$ is a point where the Nijenhuis tensor of $J$ is non-vanishing, then there is a neighbourhood of $p$ with complex coordinates $z_1, z_2$ and a holomorphic function $h(z_1)$ such that the almost-K\"aler structure on $X$ is given by Equation (\ref{eq:finalint}). 
\end{thm}

The geometry of the local normal form (\ref{eq:finalint}) is constrained enough to preclude the existence of non-trivial complete static solutions.

\begin{cor}
	A complete static solution to symplectic curvature flow in dimension four is K\"ahler-Einstein.
\end{cor}

This corollary can be compared to results by Streets--Tian \cite{StreetsTian14} and Kelleher \cite{Kelleher19} on static solutions. By contrast to their techniques, our calculations are primarily local in nature. We also note that Pook \cite{pook2012homogeneous} has constructed compact static solutions to symplectic curvature flow in dimensions $n (n+1).$

The key ingredient in the proof of Theorem \ref{thm:locnormalform} is an integrability condition that must be satisfied by any static solution. The existence of this integrability condition implies that the overdetermined PDE system describing static solutions is not involutive in the sense of exterior differential systems. It is then natural to ask if the more general soliton system is involutive. Our second main result, Theorem \ref{thm:SolitonEDS}, shows that the soliton system is involutive. Hence, the Cartan-K\"ahler Theorem may be applied to prove the existence of solutions. We summarize this result here as:
\begin{thm}
	The overdetermined PDE system for four-dimensional symplectic curvature flow solutions is involutive and symplectic curvature flow solitons exist locally.
\end{thm}

In \S\ref{eg:ASL2R} we provide an explicit example of homogeneous symplectic curvature flow soliton. Further examples have appeared in \cite{LauretWill17} and explicit solutions to symplectic curvature flow have been analyzed in \cites{LauretSymp15,pook2012homogeneous,FernandezCulma15}.

%%%%%%%%%%%%%%%%%%%%%%%%%%%%%%%%%%%%%%%%%%%%%%%%%%%%%%%%%%
\section{Structure equations}
%%%%%%%%%%%%%%%%%%%%%%%%%%%%%%%%%%%%%%%%%%%%%%%%%%%%%%%%%%

Let $X$ be a 4-manifold endowed with an almost K\"ahler structure $\left(\Omega, J \right),$ that is to say $\Omega$ is a symplectic form on $X$ and $J$ is a $\Omega$-compatible almost complex structure on $X$. Together, $\Omega$ and $J$ determine a Riemannian metric $g$ on $X.$

Let $\mathcal{P}$ denote the $\U(2)$-structure on $X$ determined by the almost K\"ahler structure $\left( \Omega, J \right)$. Explicitly, $\pi :\mathcal{P} \to X$ is the principal $\U(2)$-bundle defined by
\begin{equation*}
	\mathcal{P} = \left\lbrace u : T_p X \to \C^2 \mid \text{$u$ is a complex linear isomorphism,} \:\:\: u^* \Omega_{\mathrm{Std}} = \Omega_p \right\rbrace,
\end{equation*}
where $\Omega_{\mathrm{Std}}$ denotes the standard K\"ahler form on $\C^2$,
\begin{equation*}
	\Omega_{\mathrm{Std}} = \tfrac{i}{2} \left( e^1 \wedge \overline{e^1} + e^2 \wedge \overline{e^2} \right).
\end{equation*}

Let $\eta$ denote the $\C^2$-valued tautological 1-form on $\mathcal{P},$ defined by $\eta (v) = u (\pi_* (v)),$ for $v \in T_u \mathcal{P}.$ Denote the components of $\eta$ with respect to the standard basis $e_1,$ $e_2$ of $\C^2$ by $\eta_1, \eta_2.$ The forms $\eta_1, \eta_2$ are a basis for the semi-basic forms on $\mathcal{P},$ and they encode the almost complex structure $J$ in the sense that a 1-form $\theta$ on $X$ is a $(1,0)$-form if and only if the pullback $\pi^* \theta$ lies in $\operatorname{span}(\eta_1, \eta_2).$ We also have that, on $\mathcal{P},$
\begin{equation*}
	\begin{aligned}
		\Omega &= \tfrac{i}{2} \eta_i \wedge \overline{\eta_i}, \\
		g &= \eta_i \cdot \overline{\eta_i},
	\end{aligned}
\end{equation*}
where $1 \leq i \leq 2$ and the unitary summation convention is employed (as it will be in the remainder of this article).

\subsection{The first structure equation}

On $\mathcal{P},$ Cartan's first structure equation reads
\begin{equation}\label{eq:CartIpre}
	d \eta_i = - \kappa_{i \overline{j}} \wedge \eta_j - \xi_{ij} \wedge \overline{\eta_j},
\end{equation}
where $\kappa_{i \overline{j}} = - \overline{\kappa_{j \overline{i}}}$ and $\xi_{ij} = -\xi_{ji}.$ Here, $\kappa$ is the $\mathfrak{u}(2)$-valued connection form for the \emph{Chern connection} on $X.$ The $\Lambda^2_{\C}$-valued 1-form $\xi_{ij}$ is $\pi$-semibasic, so that there exist functions $A_{ij \overline{k}}$ and $N_{ijk}$ on $X$ such that
\begin{equation*}
	\xi_{ij} = A_{ij\overline{k}} \eta_k + N_{ijk} \overline{\eta_k}.
\end{equation*}
It is easily checked that the symplectic condition $d \Omega = 0$ implies $A_{ij\overline{k}} = 0.$ Cartan's first structure equation (\ref{eq:CartIpre}) may therefore be rewritten as
\begin{equation}\label{eq:CartI}
	d \eta_i = - \kappa_{i \overline{j}} \wedge \eta_j + N_{ijk} \, \overline{\eta_j} \wedge \overline{\eta_k}.
\end{equation}

The tensor $ N = N_{ijk} \left(\overline{\eta_i \wedge \eta_j}\right) \otimes \overline{\eta_k}$ descends to $X$ to give a well-defined section of $\Lambda^{0,2} \otimes \Lambda^{0,1}.$ In fact, $N$ is simply the \emph{Nijenhuis tensor} of the almost complex structure $J.$

The structure equations (\ref{eq:CartI}) are valid in any even dimension. However, they may be written in a simpler form in complex dimension 2, by exploiting the fact that $\Lambda^2_\mathbb{C}$ is 1-dimensional, and this simplification is worthwhile because it leads to a simplification of the second structure equation. Let $\eps_{ij}$ denote the totally skew-symmetric symbol with $\eps_{12} = 1/2.$ Let $N_i$ denote the functions on $\mathcal{P}$ defined by
\begin{equation*}
	N_{ijk} = \eps_{ij} N_k.
\end{equation*}
The first structure equation (\ref{eq:CartI}) may be rewritten as
\begin{equation}\label{eq:fourdCone}
	d \eta_i = - \kappa_{i \overline{j}} \wedge \eta_j + \eps_{ij} N_{k} \, \overline{\eta_j} \wedge \overline{\eta_k}.
\end{equation}

\subsection{The second structure equations}

The identity $d^2 = 0$ applied to equation (\ref{eq:fourdCone}) implies the following equations:
\begin{equation}\label{eq:structtwo}
	\begin{aligned}
		d N_i =& \overline{\eps_{jk}} A_{ij} \eta_k + B \, \eta_i + \left(F_{ij} + \eps_{ij} H \right) \overline{\eta_j} - \kappa_{i \overline{j}} N_j - \kappa_{j \overline{j}} N_i, \\
		d \kappa_{i \overline{j}} =& - \kappa_{i \overline{k}} \wedge \kappa_{k \overline{j}} + K_{i\bar{j}k\bar{l}} \overline{\eta_k} \wedge \eta_l + \left(R  + N_l \overline{N_l} \right) \left(-\tfrac{1}{3} \eta_i \wedge \overline{\eta_j} - \tfrac{1}{3} \delta_{i \bar{j}} \eta_k \wedge \overline{\eta_k} \right) \\
		& + N_i \overline{N_j} \eta_k \wedge \overline{\eta_k} + Q_{i \bar{k}} \eta_k \wedge \overline{\eta_j} + Q_{k \bar{j}} \eta_i \wedge \overline{\eta_k} - \tfrac{1}{2} A_{ik} \overline{\eta_j \wedge \eta_k} + \tfrac{1}{2} \overline{A_{jk}} \eta_i \wedge \eta_k \\
		& + 2 B \, \eps_{ik} \overline{\eta_j \wedge \eta_k} - 2 \overline{B} \overline{\eps_{jk}} \eta_i \wedge \eta_k,
	\end{aligned}
\end{equation}
for functions $A_{ij}, B, F_{ij}, H, K_{i \bar{j} k \bar{l}}, Q_{i \bar{j}}, R$ on $\mathcal{P}$ having the following symmetries:
\begin{equation*}
	\begin{aligned}
		A_{ij} & = A_{ji}, & K_{i \bar{j} k \bar{l}} &= K_{k \bar{j} i \bar{l}} & Q_{i \bar{j}} &= -\overline{Q_{j \bar{i}}} \\
		F_{ij} & = F_{ji}, & K_{i \bar{j} k \bar{l}} &= K_{i \bar{l} k \bar{j}} & Q_{i \bar{i}} &= 0 \\
		 & & K_{i \bar{i} k \bar{l}} &= 0 & & \\
		 & & K_{i \bar{j} k \bar{l}} &= \overline{K_{j \bar{i} l \bar{k}}} & &  
	\end{aligned}
\end{equation*}
Each of these functions takes values in an irreducible $\U(2)$-representation, and the right-hand-side of the equation for $d \kappa$ in (\ref{eq:structtwo}) represents the irreducible decomposition of the curvature of the Chern connection of $X.$ The equations (\ref{eq:structtwo}) are called the second structure equations of the almost K\"ahler structure $(g, \Omega).$ By a classical theorem of Cartan, the functions $A_{ij}, B, F_{ij}, K_{i \bar{j} k \bar{l}}, Q_{i \bar{j}}, R$ form a complete set of second-order invariants of $(g, \Omega).$

The second equation of (\ref{eq:structtwo}) gives the curvature of the Chern connection on $X.$ The Chern-Ricci form is the trace of this curvature:
\begin{equation*}
	d \kappa_{i \overline{i}} = -R \eta_{i} \wedge \overline{\eta_i} - 4 i \Im \left(\overline{B} \eta_1 \wedge \eta_2 \right) - 2 Q_{i \overline{j}} \overline{\eta_i} \wedge \eta_j.
\end{equation*}

\subsubsection{The curvature of $g$}

The Riemann curvature tensor of a 4-dimensional manifold is a section of a vector bundle modeled on the $\SO(4)$ representation
\begin{equation*}\label{eq:curvdecompso}
	\mathrm{Sym}^2 \left( \Lambda^2 \R^4 \right) \cong \R \oplus \mathrm{Sym}^2_0 \R^4 \oplus \mathrm{Sym}^2_0 \left(\Lambda^2_+ \R^4 \right) \oplus \mathrm{Sym}^2_0 \left( \Lambda^2_- \R^4 \right).
\end{equation*}
Corresponding to each irreducible component, we have the scalar curvature $r$, the traceless Ricci curvature $\mathrm{Ric}^0$, and the self-dual $W^+$ and anti-self dual $W^-$ Weyl tensors respectively.

%The irreducible decompositions of the spaces on the right hand side of (\ref{eq:curvdecompso}) under the action of $U(2)$ are:
%\begin{equation*}
%	\begin{aligned}
%		\begin{aligned}
%			\mathrm{Sym}^2_0 \R^4 &\cong \mathrm{Sym}^2_\C \C^2 \oplus \left(\C^2 \oplus \overline{\C^2} \right)_0, \\
%			\mathrm{Sym}^2_0 \left(\Lambda^2_+ \R^4 \right) &\cong \cdots \\
%			\mathrm{Sym}^2_0 \left(\Lambda^2_- \R^4 \right) &\cong \cdots
%		\end{aligned}
%	\end{aligned}
%\end{equation*}

The second structure equations (\ref{eq:structtwo}) can be compared with the structure equations of a Riemannian manifold to write the components of the Riemann curvature tensor of $g$ in terms of the first and second order invariants of the almost K\"ahler structure. The result of this calculation is:
\begin{equation*}
	\begin{aligned}
		\mathrm{Scal}(g) &= - 8 N_i \overline{N_i} - 8 R, \\
		\mathrm{Ric}(g) &= A_{ij} \overline{\eta_i} \cdot \overline{\eta_j} + \overline{A_{ij}} \eta_i \cdot \eta_j + \left( Q_{i \bar{j}} + N_i \overline{N_j} \right) \overline{\eta_i} \cdot \eta_j - \left(\tfrac{1}{2} R + N_k \overline{N_k} \right) \eta_i \cdot \overline{\eta_i}, \\
		W^+ (g) &= \left(4 N_i \overline{N_i} - 4 R \right) \Omega^2 - 8 i B \, \Omega \cdot \left( \eta_1 \wedge \eta_2 \right) + 8 i \overline{B} \, \Omega \cdot \left( \overline{\eta_1 \wedge \eta_2 } \right) \\
		& + 2 H \, \left(\eta_1 \wedge \eta_2 \right)^2 + 2 \overline{H} \left( \overline{\eta_1 \wedge \eta_2 } \right)^2 - 4 N_i \overline{N_i} \left(\eta_1 \wedge \eta_2 \right) \cdot \left(\overline{\eta_1 \wedge \eta_2 }\right), \\
		W^-(g) &=  -2 K_{i \bar{j} k \bar{l}} \left(\overline{\eta}_i \wedge \eta_j \right) \cdot \left(\overline{\eta_k} \wedge \eta_l \right),
	\end{aligned}
\end{equation*}
where here we are viewing the Ricci tensor as a symmetric 2-tensor, and the self-dual and anti-self-dual Weyl tensors as sections of $\mathrm{Sym}^2_0 \left(\Lambda^2_{\pm} T X \right).$

\subsubsection{Symplectic curvature flow}

The right hand side of the symplectic curvature flow equation (\ref{eq:introsympcurvfl}) may also be written in terms of second and first order invariants of $\left(g, \Omega \right).$ We have
\begin{equation}\label{eq:symplcurvfloweqs}
	\begin{aligned}
		\frac{\partial}{\partial t} \Omega &= 4 R \, \Omega + 4 i Q_{i \bar{j}} \, \overline{\eta_i} \wedge \eta_{j} - 8 \, \Im \left(\overline{B} {\eta_1 \wedge \eta_2 } \right), \\
		\frac{\partial}{\partial t} g & =  4 R \, g - 8 Q_{i \overline{j}} \overline{\eta_i} \cdot \eta_j - 4 \Re ( A_{ij} \overline{\eta_i} \cdot \overline{\eta_{j}}), \\
		\frac{\partial}{\partial t} J &= \left( 8 \Im \left( \overline{B} {\eta_1 \wedge \eta_2 } \right) - 4 \Re \left( A_{ij} \overline{\eta_i} \wedge \overline{\eta_j} \right) \right) g^{-1},
	\end{aligned}
\end{equation}

\section{Static solutions in dimension four}\label{sect:staticsolns}

We now study static solutions to symplectic curvature flow in dimension 4. These are solutions which evolve strictly by rescaling under the flow, so for such a solution we must have
\begin{equation}
	\begin{aligned}
		\frac{\partial}{\partial t} \Omega &= \lambda \Omega, \\
		\frac{\partial}{\partial t} g &= \lambda g, \\
		\frac{\partial}{\partial t} J &= 0.
	\end{aligned}
\end{equation}
Comparing with (\ref{eq:symplcurvfloweqs}), we see that the static almost K\"ahler structures are characterized by the following conditions on their second order invariants:
\begin{equation}\label{eq:statinvcond}
	\begin{aligned}
		R = \frac{\lambda}{4}, \:\:\: B = 0, \:\:\: Q_{i \bar{j}} = 0, \:\:\: A_{ij} = 0.
	\end{aligned}
\end{equation}
Therefore, the first and second structures equations for a static solution reduce to
\begin{equation*}
	\begin{aligned}
		d \eta_i &= - \kappa_{i \overline{j}} \wedge \eta_j + \eps_{ij} N_{k} \, \overline{\eta_j} \wedge \overline{\eta_k}, \\
		d N_i =& \left(F_{ij} + \eps_{ij} H \right) \overline{\eta_j} - \kappa_{i \overline{j}} N_j - \kappa_{j \overline{j}} N_i, \\
		d \kappa_{i \overline{j}} =& - \kappa_{i \overline{k}} \wedge \kappa_{k \overline{j}} + K_{i\bar{j}k\bar{l}} \overline{\eta_k} \wedge \eta_l + \left(\tfrac{\lambda}{4}  + N_l \overline{N_l} \right) \left(-\tfrac{1}{3} \eta_i \wedge \overline{\eta_j} - \tfrac{1}{3} \delta_{i \bar{j}} \eta_k \wedge \overline{\eta_k} \right) \\
		& + N_i \overline{N_j} \eta_k \wedge \overline{\eta_k} + Q_{i \bar{k}} \eta_k \wedge \overline{\eta_j}.
	\end{aligned}
\end{equation*}
Differentiating the second equation, we find
\begin{equation*}\label{eq:firstcondition}
	0 = d^2 N_i \wedge \overline{\eta_1 \wedge \eta_2} = \left(F_{ij} \overline{N_j} + \frac{1}{2} \eps_{ij} H \overline{N_j} \right) \eta_1 \wedge \eta_2 \wedge \overline{\eta_1 \wedge \eta_2}.
\end{equation*}
The equations
\begin{equation}\label{eq:firsttorscond}
	F_{ij} \overline{N_j} + \frac{1}{2} \eps_{ij} H \overline{N_j} = 0, \:\:\: i = 1, 2
\end{equation}
must therefore be satisfied by any static almost K\"ahler structure. These equations give a restriction on the second-order invariants of a static solution which is not an algebraic consequence of (\ref{eq:statinvcond}), so the static equations (\ref{eq:statinvcond}) are not involutive in the sense of exterior differential systems.

The equation (\ref{eq:firsttorscond}) may be simplified by adapting coframes. Suppose $N$ is non-zero at a point $p \in X$. Say a coframe $u \in \mathcal{P}_p$ is $N$-adapted if $N_2 = 0$ at $u \in \mathcal{P}.$ The group $\U(2)$ acts transitively on $\Lambda^{0,2} \otimes \Lambda^{0,1}$ with stabilizer $\U(1) \times \U(1),$ so the collection of all $N$-adapted coframes is a $\U(1) \times \U(1)$-bundle over the locus in $X$ where $N \neq 0.$ For simplicity, we shall assume from now on $N$ is nowhere vanishing on $X$ (otherwise restrict to the open dense locus where $N \neq 0$). Denote the bundle of $N$-adapted coframes by $\mathcal{P'} \to X.$

Equations (\ref{eq:firsttorscond}) imply that, on $\mathcal{P}',$
\begin{equation*}
	F_{11} = 0, \:\:\: F_{12} = \tfrac{1}{2} H.
\end{equation*}
Differentiating the identity $N_2 = 0,$ we find that on $\mathcal{P}'$
\begin{equation*}
	0 = F_{22} \overline{\eta_2} - N_1 \kappa_{2 \bar{1}}.
\end{equation*}
Define $G = N_1 F_{22},$ so we have $\kappa_{2 \bar{1}} = G \, \overline{\eta_2}.$ Let us also define $\R$-valued 1-forms $\alpha = - i \kappa_{1 \bar{1}}$ and $\beta = - i \kappa_{2 \bar{2}}.$ The forms $\alpha$ and $\beta$ together define a connection on the $\U(1) \times \U(1)$-bundle $\mathcal{P}'.$ The structure equations restricted to $\mathcal{P}'$ now read
\begin{equation*}
	\begin{aligned}
		d \eta_1 =& i \alpha \wedge \eta_1 + N_1 \overline{\eta_1 \wedge \eta_2}, \\
		d \eta_2 =& i \beta \wedge \eta_2 + G \eta_1 \wedge \overline{\eta_2}, \\
		d N_1 =& H \overline{\eta_2} + 2 i N_1 \alpha + i N_1 \beta, \\
		d \alpha =& i \left(\tfrac{1}{3} \lvert N_1 \rvert^2 - \tfrac{1}{6} \lambda - K_{1 \bar{1} 1 \bar{1}} \right) \eta_1 \wedge \overline{\eta_1} + i K_{2 \bar{2} 2 \bar{1}} \eta_1 \wedge \overline{\eta_2} - i K_{1 \bar{1} 1 \bar{2}} \eta_2 \wedge \overline{\eta_1} \\
		& + i \left(\tfrac{2}{3} \lvert N_1 \rvert^2 - \tfrac{1}{12} \lambda + \lvert G \rvert^2 + K_{1 \bar{1} 1 \bar{1}} \right) \eta_2 \wedge \overline{\eta_2}, \\
		d \beta =& i \left(-\tfrac{1}{3} \lvert N_1 \rvert^2 - \tfrac{1}{12} \lambda + K_{1 \bar{1} 1 \bar{1}} \right) \eta_1 \wedge \overline{\eta_1} - i K_{2 \bar{2} 2 \bar{1}} \eta_1 \wedge \overline{\eta_2} + i K_{1 \bar{1} 1 \bar{2}} \eta_2 \wedge \overline{\eta_1} \\
		& + i \left(-\tfrac{2}{3} \lvert N_1 \rvert^2 - \tfrac{1}{6} \lambda - \lvert G \rvert^2 - K_{1 \bar{1} 1 \bar{1}} \right) \eta_2 \wedge \overline{\eta_2}
	\end{aligned}
\end{equation*}

The identities $d^2 \eta_i = 0$ imply
\begin{equation*}
	\begin{aligned}
		K_{1 \bar{1} 1 \bar{2}} &= 0, \:\:\:\: K_{2 \bar{2} 2 \bar{1}} = 0, \\
		K_{1 \bar{1} 1 \bar{1}} &= \tfrac{1}{3} \lvert N_1 \rvert^2 + \tfrac{1}{12} \lambda - \lvert G \rvert^2, \\
		d \, G &= G_{\bar{1}} \eta_1 + G_{2} \overline{\eta_2} - i G \alpha + 2 i G \beta,
	\end{aligned}
\end{equation*}
for some $\C$-valued functions $G_{\bar{1}}$ and $G_2$ on $\mathcal{P}'.$ The identity $d^2 \alpha = 0$ implies $G G_2 = 0,$ so $G_2$ must vanish identically on $\mathcal{P}'$ (since the vanishing of $G$ implies the vanishing of $G_2$).

Next, differentiating the identity $\kappa_{2 \bar{1}} = G \, \overline{\eta_2}$ implies
\begin{equation*}
	K_{1\bar{2}1\bar{2}} = - \overline{G_{\bar{1}}}, \:\:\:\:\: K_{2\bar{1}2\bar{1}} = - G_{\bar{1}}.
\end{equation*}
The identity $d^2 G = 0$ yields
\begin{equation*}
	N_1 G_{\bar{1}} \eta_1 \wedge \overline{\eta_1 \wedge \eta_2} + G \left( 3 \lvert N_1 \rvert^2 - \tfrac{1}{2} \lambda \right) \overline{\eta_2} \wedge \eta_1 \wedge \eta_2 = 0,
\end{equation*}
leading to two possibilities:
\begin{enumerate}
	\item $G$ vanishes identically on $\mathcal{P}$; or
	\item $\lvert N_1 \rvert^2 = \tfrac{1}{6} \lambda$ and $G_{\bar{1}} = 0$ on $\mathcal{P}'.$
\end{enumerate}
Let us suppose case (2) holds. Differentiating $\lvert N_1 \rvert^2 = \tfrac{1}{6} \lambda$ implies $H = 0$ on $\mathcal{P},$ so we have
\begin{equation*}
	d N_1 = 2 i N_1 \alpha + i N_1 \beta.
\end{equation*}
The identity $d^2 N_1$ then implies $\lvert N_1 \rvert^2 = 0,$ contradicting our assumption that $N$ is non-zero. Therefore case (2) is not possible and we must instead have case (1) holding: $G = 0$ on $\mathcal{P}'.$

Finally, the identity $d^2 N_1 \wedge \overline{\eta_2} = 0$ implies that $\lambda N_1 = 0,$ so we must have $\lambda = 0$ for any non-trivial static solution to symplectic curvature flow. 

At this stage, the structure equations have simplified to
\begin{equation}\label{eq:finishedstructeqs}
	\begin{aligned}
		d \eta_1 =& i \alpha \wedge \eta_1 + N_1 \overline{\eta_1 \wedge \eta_2}, \\
		d \eta_2 =& i \beta \wedge \eta_2, \\
		d N_1 =& H \overline{\eta_2} + 2 i N_1 \alpha + i N_1 \beta, \\
		d \alpha =& i \lvert N_1 \rvert^2 \eta_2 \wedge \overline{\eta_2}  \\
		d \beta =& -i \lvert N_1 \rvert^2 \eta_2 \wedge \overline{\eta_2}.
	\end{aligned}
\end{equation}

It is easy to check that equations (\ref{eq:finishedstructeqs}) represent an involutive prescribed coframing problem in the sense of \cite{BryEDSNotes}. The primary invariants are the real and imaginary parts of the function $N_1,$ the free derivatives are the real and imaginary parts of $H$ and the tableau of free derivatives is equivalent to the standard Cauchy-Riemann tableau
\begin{equation*}
	\begin{bmatrix}
		x & y \\
		-y & x
	\end{bmatrix}.
\end{equation*} 
Therefore, Theorem 3 of \cite{BryEDSNotes} may be applied to show that non-trivial static solutions to symplectic curvature flow exist locally and depend on two functions of one variable.

We summarize the conclusions of this section in the following theorem.

\begin{thm}
	A non-trivial static solution to symplectic curvature flow in dimension 4 must have $\lambda = 0.$ Real analytic non-trivial static solutions exist locally and depend on two functions of one variable in the sense of exterior differential systems.
\end{thm}

\begin{remark}
	We will show in the following subsection that any static solution must be real analytic.
\end{remark}

\subsubsection{Integrating the structure equations}

In this subsection we shall integrate the structure equations (\ref{eq:finishedstructeqs}) to obtain a local normal form for four-dimensional static solutions to symplectic curvature flow. This local normal form will allow us to draw conclusions on the global structure of such solutions.

We begin by noting that $d \left(\alpha + \beta \right) = 0,$ so we have locally $\beta = -\alpha + d g$ for some function $g.$ The function $g$ may be integrated away in the structure equations, so we may assume we are working on the $\U(1)$-subbundle $\mathcal{P}'' \to X$ where $\beta = -\alpha.$

The distribution $\eta_1 = 0$ descends from $\mathcal{P}''$ to $X$ to give a well-defined real codimension-two distribution on $X.$ Each leaf of the resulting foliation has a metric given by the restriction of $\left\lvert N_1 \right\rvert^2 \overline{\eta_2} \cdot \eta_2,$ and the equations
\begin{equation*}
	\begin{aligned}
		d \left(N_1 \overline{\eta_2}\right) &= 2 i \alpha \wedge \left(N_1 \overline{\eta_2}\right), \\
		d \alpha &= -i \left(N_1 \overline{\eta_2}\right) \wedge \overline{\left(N_1 \overline{\eta_2}\right)}
	\end{aligned}
\end{equation*}
imply that this metric has constant curvature $-4.$ It follows that if $U \subset X$ is simply-connected, then there exists a $\C$-valued function $z_1$ and an $\R$-valued function $s$ on $\mathcal{P''}|_U$ such that
\begin{equation*}
	N_1 \, \overline{\eta_2} = \frac{e^{is} \, d z_1}{1 - \left\lvert z_1 \right\rvert^2}, \:\:\:\: \alpha = -\frac{i}{2} \frac{\overline{z_1} d z_1 - z_1 d \overline{z_1}}{1 - \left\lvert z_1 \right\rvert^2} + ds.
\end{equation*}
We restrict to the locus $s = 0.$ This amounts to restricting the $\U(1)$-structure $\mathcal{P}''$ to an $\left\lbrace e \right\rbrace$-structure over $U$. 

The equation
\begin{equation*}
	d \left( N_1 d z_1 \right) = \frac{1}{2} \frac{N_1 z_z}{1 - \lvert z_1 \rvert^2} d \overline{z_1} \wedge d z_1
\end{equation*}
implies
\begin{equation}\label{eq:nijint}
	N_1 = \frac{h(z_1)}{\sqrt{1 - \lvert z_1 \rvert^2}}
\end{equation}
for some holomorphic function $h$ of a single complex variable. The $d \eta_1$ equation in (\ref{eq:finishedstructeqs}) implies
\begin{equation*}
	d \begin{bmatrix}
		\eta_1 \\
		\overline{\eta_1}
	\end{bmatrix} = - \frac{1}{2} \frac{1}{1 - \lvert z_1 \rvert^2} \begin{bmatrix}
	 \overline{z_1} d z_1 - z_1 d \overline{z_1} & 2 \, d z_1 \\
	2 \, d \overline{z_1} & z_1 d \overline{z_1} - \overline{z_1} d z_1
\end{bmatrix} \wedge \begin{bmatrix}
\eta_1 \\
\overline{\eta_1}
\end{bmatrix}.
\end{equation*}
It follows that
\begin{equation*}
	d \left( \frac{\eta_1 + z_1 \overline{\eta_1}}{\sqrt{1-\lvert z_1 \rvert^2}} \right) = 0,
\end{equation*}
so there exists a coordinate $z_2$ on $U$ with
\begin{equation*}
	\eta_1 = \frac{d z_2 - z_1 d \overline{z_2}}{\sqrt{1-\lvert z_1 \rvert^2}}.
\end{equation*}
The coordinate $z_2$ is unique up to addition of a constant.

We have now proven the first part of the following theorem. The second part follows by reversing the steps above.

\begin{thm}\label{thm:locnormalform}
	Let $\left(X, \Omega, g \right)$ be a non-trivial 4-dimensional static solution to symplectic curvature flow and suppose $p \in X$ is a point where the Nijenhuis tensor is non-vanishing. Then there is a local neighbourhood of $p$ with complex coordinates $z_1$ and $z_2$ and a holomorphic function $h(z_1)$ such that the symplectic form $\Omega$ and metric $g$ on $X$ are given by
	\begin{equation}\label{eq:finalint}
		\begin{aligned}
			\Omega &= \frac{i}{2} \left( \frac{d z_1 \wedge d \overline{z_1} }{\lvert h(z_1) \rvert^2 \left(1 - \lvert z_1 \rvert^2 \right)} + d z_2 \wedge d \overline{z_2} \right), \\
			g &= \frac{d z_1 \cdot d \overline{z_1} }{\lvert h(z_1) \rvert^2 \left(1 - \lvert z_1 \rvert^2 \right)} + \frac{1 + \lvert z_1 \rvert^2}{1 - \lvert z_1 \rvert^2} d z_2 \cdot d \overline{z_2} - \Re \left(\frac{\overline{z_1} d z_2^2}{1 - \lvert z_1 \rvert^2}\right).
		\end{aligned}
	\end{equation}

Conversely, let $h(z_1)$ be a meromorphic function on the unit disk and let $\Sigma \subset \C$ be the subset of the unit disk where $h(z_1)$ has no zeros or poles. Then equation (\ref{eq:finalint}) defines a static solution to symplectic curvature flow on the 4-manifold $\Sigma \times \C.$
\end{thm}

The local normal form of Theorem \ref{thm:locnormalform} is strong enough to derive global consequences.

\begin{cor}
	There are no non-trivial complete static solutions to symplectic curvature flow in dimension 4.
\end{cor}

\begin{proof}
	Let $L$ be a leaf of the foliation $\eta_1 = 0$ on $X$ and let $\widehat{L}$ denote the universal cover of $L.$ The function $z_1 : \widehat{L} \to \mathbb{C}$ and holomorphic function $h(z_1)$ exist globally on $\widehat{L},$ and $z_1$ identifies $\widehat{L}$ with the unit disk $\mathbb{D} \subset \mathbb{C}.$ Formula (\ref{eq:finalint}) implies that the boundary $\left\lvert z_1 \right\rvert = 1$ is at finite distance, and the function $\left\lvert N_1 \right\rvert^2$ blows up there. Therefore the metric induced on $\widehat{L}$ is incomplete, hence the metric $g|_L$ is incomplete, hence $g$ is incomplete.
\end{proof}

\section{Local existence of soliton solutions}

In the previous section we showed that the static system was non-involutive in the sense of exterior differential systems. This non-involutivity lead to several compatibility conditions which placed severe restrictions on the local geometry of a static solution. By contrast, the soliton equation is involutive as shall be explained in this section. The upshot is a local existence and uniqueness theorem for soliton solutions.

The soliton equations are
\begin{equation}
	\begin{aligned}
		\frac{\partial}{\partial t} \Omega &= \lambda \Omega + \mathcal{L}_V \Omega, \\
		\frac{\partial}{\partial t} g &= \lambda g + \mathcal{L}_V g,
	\end{aligned}
\end{equation}
where $\lambda \in \R$ is a constant and $V$ is a vector field on $X.$

Let $V$ be a vector field on an almost-K\"ahler 4-manifold $X,$ thought of as a $\U(2)$-equivariant map $\mathcal{B} \to \C^2$ and let $V_i$ be the components of this map with respect to the standard basis of $\C^2.$ There exist functions $U,$ $S_{ij},$ $W_{i \bar{j}},$ and $Y$ with symmetries $W_{i \bar{i}} = 0,$ $S_{ij} = S_{ji}$ such that, on $\mathcal{B},$
\begin{equation}\label{eq:dVectorField}
	d V_i + \kappa_{i \overline{j}} V_j = \left( W_{i \bar{j}} + Y \delta_{i \bar{j}} \right) \eta_j +  \left( S_{ij} + U \epsilon_{ij} \right) \overline{\eta_j}.
\end{equation}
These functions appear in the Lie derivatives of $\Omega$ and $g$ with respect to $V$:
\begin{equation*}
	\begin{aligned}
		\mathcal{L}_V \Omega &= - \tfrac{i}{2} \left(W_{i \bar{j}} + \overline{W_{j \bar{i}}} \right) \overline{\eta_{i}} \wedge \eta_j + 2\, \Re \left(U \right) \Omega + \Im \left( \left( Y - V_i \overline{N_i} \right) \eta_1 \wedge \eta_2 \right) \\
		\mathcal{L}_V g & = 2\, \Re \left( \overline{S_{ij}} \eta_i \cdot \eta_j \right) + 4 \Re \left( \overline{\epsilon_{kj}} V_k \overline{N}_i \eta_i \cdot \eta_j \right) + \left(W_{i \bar{j}} + \overline{W}_{j \bar{i}} \right) + 2 \, \Re (U) g \overline{\eta}_i \cdot \eta_j.
	\end{aligned}
\end{equation*}

 Comparing with the symplectic curvature flow equations (\ref{eq:symplcurvfloweqs}), we see that the second-order invariants of a soliton solution satisfy
 \begin{equation}\label{eq:SolConds}
 	\begin{aligned}
 		Q_{i \bar{j}} &= -\tfrac{1}{4} \Re (W_{i \bar{j}}), & R &= \tfrac{1}{2} \Re(Y) + \tfrac{1}{4} \lambda, \\
 		A_{ij} &= -\tfrac{1}{2} S_{ij} + \tfrac{1}{2} \eps_{ik} \overline{V_k} N_j + \tfrac{1}{2} \eps_{jk} \overline{V_k} N_i, & B &= \tfrac{1}{8} U + \tfrac{1}{8} \overline{V_i} N_i.
 	\end{aligned}
 \end{equation}
Conversely, if $(X, \Omega, J)$ is an almost-K\"ahler 4-manifold with a vector field $V$ so that equations (\ref{eq:fourdCone}), (\ref{eq:structtwo}), (\ref{eq:dVectorField}), (\ref{eq:SolConds}) are satisfied on the $\U(2)$-bundle $\mathcal{B},$ then $(X, \Omega, J)$ is a soliton for symplectic curvature flow.

\begin{thm}\label{thm:SolitonEDS}
	Solitons for the symplectic curvature flow exist locally and depend on 10 functions of 3 variables in the sense of exterior differential systems.
\end{thm}

\begin{proof}
	We first reduce the problem on constructing symplectic curvature flow solitons to a prescribed coframing problem in the style of \cite{BryEDSNotes}.
	
	By the above paragraph, the $\U(2)$-coframe bundle $\mathcal{B} \to X$ of a symplectic curvature soliton carries 1-forms $\eta$ and $\kappa$ and functions $N_i,$ $V_i,$ $W_{i\bar{j}}$, $Y$, $S_{ij},$ $U,$  $F_{ij}$ and $K_{i \bar{j} k \bar{l}}$ satisfying equations  (\ref{eq:fourdCone}), (\ref{eq:structtwo}), (\ref{eq:dVectorField}), (\ref{eq:SolConds}). Conversely, if $M$ is an 8-manifold together with 1-forms $\eta_i$ and $\kappa_{i \bar{j}}$ and functions $N_i,$ $V_i,$ $W_{i\bar{j}}$, $Y$, $S_{ij},$ $U,$  $F_{ij}$ and $K_{i \bar{j} k \bar{l}}$ satisfying equations  (\ref{eq:fourdCone}), (\ref{eq:structtwo}), (\ref{eq:dVectorField}), (\ref{eq:SolConds}), then an argument similar to Theorem 1 of \cite{BryantBochKah01} implies that (after possibly shrinking $M$) the 4-dimensional leaf space $X$ of the integrable plane field $\eta = 0$ carries an almost-K\"ahler structure which is a soliton for the symplectic curvature flow and $M$ may be identified an an open set in the $\U(2)$-bundle $\mathcal{B} \to X$ associated to this almost-K\"ahler structure.
	
	The prescribed coframing problem (\ref{eq:fourdCone}), (\ref{eq:structtwo}), (\ref{eq:dVectorField}), (\ref{eq:SolConds}) is written in a form where it is natural to attempt to apply Theorem 3 of \cite{BryEDSNotes}. The `primary invariants' are the functions $N_i,$ $V_i,$ $\Re(Y),$ $\Re(W_{i \bar{j}}),$ $U,$ $S_{ij},$ and $K_{i \bar{j} k \bar{l}}$ while the derived invariants consist of the covariant derivatives of these functions with respect to $\kappa_{i \bar{j}},$ taking the identity $d^2 \eta = 0$ into account. However, the Cartan's existence theorem cannot be applied directly because the tableau of free derivatives is not involutive. One may compute that it has Cartan characters $(24,22,13,1),$ while the dimension of the prolongation is $106,$ so Cartan's test fails. Nevertheless, the system does become involutive after one prolongation. We omit the details here due to length, but after prolongation the tableau of free derivatives has Cartan characters $(56,31,10,0)$ and the dimension of the prolongation is $148 = (56)+2(31)+3(10)+4(0),$ so Cartan's test is passed and the system is involutive. Existence and generality in the real analytic category then follows from Theorem 3 of \cite{BryEDSNotes}.
\end{proof}

\begin{remark}[Real analyticity]
	Streets--Tian \cite{StreetsTian14} prove that symplectic curvature flow is parabolic modulo diffeomorphism. Similarly, it can be shown that the soliton system is elliptic modulo diffeomorphism. Therefore, symplectic curvature flow solitons which are $C^{1, \alpha}$ in harmonic coordinates for some $\alpha > 0$ must be real analytic in those coordinates.
\end{remark}

\begin{remark}[The gradient case]
	The equation $d V^\flat = 0$ implies
	\begin{equation*}
		\begin{aligned}
			U & = \overline{V_i} N_i, &	\Im(Y) &= 0, & \Im(W_{i \bar{j}}) &= 0. & &
		\end{aligned}
	\end{equation*}
	Adjoining these equations to (\ref{eq:fourdCone}), (\ref{eq:structtwo}), (\ref{eq:dVectorField}), (\ref{eq:SolConds}) gives a prescribed coframing problem whose local solutions correspond to gradient solitons. However, in contrast to the general case, this system is not involutive even after prolongation because the equation $d^2 Y = 0$ yields a restriction on the 3-jet of a solution. In the language of exterior differential systems, this problem has \emph{intrinsic torsion}. Unfortunately, the restriction on the 3-jet is more algebraically complicated than the equations encountered in \S\ref{sect:staticsolns} and the existence and local generality of gradient solitons remains unknown.
\end{remark}

\begin{remark}[Comparison to Laplacian flow]
	Symplectic curvature flow has a formal similarity to the Laplacian flow of closed $\G_2$-structures in that both are flows of geometric structures with torsion defined by closed differential forms. This formal similarity extends to the local existence theory for static and soliton solutions. The static solutions to Laplacian flow are the \emph{eigenforms}, analyzed in \cite{BallQuad23} where it was shown that the relevant EDS is not involutive. By contrast, the general Laplace soliton system is well-behaved. In a recent paper Bryant \cite{bryant2023generality} has shown that the EDS describing Laplace solitons is involutive.
\end{remark}

\subsection{An example}\label{eg:ASL2R}
	Let $\mathrm{G} = \mathrm{SL}_2 \R \ltimes \R^2$ denote the group of volume preserving affine transformations of $\R^2$ and write the left-invariant Maurer-Cartan form $\mu$ of $\G$ as
	\begin{equation*}
		\mu = \begin{bmatrix}
			0 & 0 & 0 \\
			\alpha_1 & \alpha_3 & \beta - \alpha_4 \\
			\alpha_2 & - \beta - \alpha_4 & -\alpha_3
		\end{bmatrix}.
	\end{equation*}
	Let $\mathrm{S}^1$ denote the circle subgroup of $\G$ generated by the action of the vector field dual to $\beta.$ For each pair of non-zero numbers $(a,b) \in \R^2$ the 2-form $\Omega_{a,b}$ and metric $g_{a,b}$ defined by
	\begin{equation*}
		\begin{aligned}
			\Omega_{a,b} &= a^2 \, \alpha_1 \wedge \alpha_2 + b^2 \, \alpha_3 \wedge \alpha_4, \\
			g_{a,b} &= a^2 \left( \alpha_1^2 + \alpha_2^2 \right) + b^2 \left( \alpha_3^2 + \alpha_4^2 \right),
		\end{aligned}
	\end{equation*}
	descend to the 4-dimensional quotient $X = \G / \mathrm{S}^1$ to define an almost K\"ahler structure on $X.$ The $U(2)$-invariants of this structure may be computed via the Maurer-Cartan equation $d \mu = - \mu \wedge \mu.$ For this structure, we find that the right hand side of the symplectic curvature flow equations are given by
	\begin{equation*}
		\begin{aligned}
			4 R \, \Omega + 4 i Q_{i \bar{j}} \, \overline{\eta_i} \wedge \eta_{j} - 8 \, \Im \left(\overline{B} {\eta_1 \wedge \eta_2 } \right) &= 4 \, \alpha_3 \wedge \alpha_4, \\
			4 R \, g - 8 Q_{i \overline{j}} \overline{\eta_i} \cdot \eta_j - 4 \Re ( A_{ij} \overline{\eta_i} \cdot \overline{\eta_{j}}) &= 4 \left(\alpha_3^2 + \alpha_4^2 \right).
		\end{aligned}
	\end{equation*}
	Therefore, the 1-parameter family of almost-K\"ahler structures on $X$ defined by
	\begin{equation*}
		\begin{aligned}
			\Omega(t) &= a^2 \, \alpha_1 \wedge \alpha_2 + \sqrt{4 t + b^4} \, \alpha_3 \wedge \alpha_4, \\
			g(t) &= a^2 \left(\alpha_1^2 + \alpha_2^2 \right) + \sqrt{4 t + b^4} \left(\alpha_3^2 + \alpha_4^2 \right).
		\end{aligned}
	\end{equation*}
	gives a solution to symplectic curvature flow with initial condition $\Omega(0) = \Omega_{a,b},$ $g(0) = g_{a,b}.$ For any $a,$ $a'$ the structures $(\Omega_{a,b}, g_{a,b})$ and $(\Omega_{a',b}, g_{a',b})$ are diffeomorphism equivalent, so in fact each $(\Omega(t), g(t))$ defines a soliton solution to symplectic curvature flow.
	
	\begin{remark}
		The manifold $X$ defined above is non-compact and $\G$ does not admit a uniform lattice, so it has no compact quotients. However, $\G$ does admit cocompact lattices $\Gamma$ (for example $\Gamma = \mathrm{SL}_2 \mathbb{Z} \ltimes \mathbb{Z}^2$) and these give quotients $\Gamma \backslash X$ with finite volume. The 1-parameter family of almost-K\"ahler structures $(\Omega(t), g(t))$ descends to each $\Gamma \backslash X$ to give a solution of symplectic curvature flow, however the diffeomorphism between the structures $(\Omega_{a,b}, g_{a,b})$ and $(\Omega_{a',b}, g_{a',b})$ does not descend to $\Gamma \backslash X,$ so $(\Omega(t), g(t))$ does not give a soliton solution on $\Gamma \backslash X.$
	\end{remark}

\bibliography{SympCurvStatRefs}

\end{document}